\documentclass[12pt]{amsart}
\usepackage[a4paper,margin=3cm]{geometry}
\usepackage[cp1251]{inputenc}
\usepackage[T2A]{fontenc}
\usepackage{amsmath,amssymb,amsthm}
\usepackage{enumerate}

\newcommand{\R}{{\mathbb R}}

\newcommand{\abs}[1]{\left\vert#1\right\vert}
\newcommand{\pair}[1]{\left\langle#1\right\rangle}

\newcommand{\set}[1]{\left\{#1\right\}}
\newcommand{\ind}[1]{\mathbf{1}_{#1}}
\newcommand{\ex}[1]{\mathsf{E}\left[#1\right]}

\numberwithin{equation}{section} \theoremstyle{plain}
\newtheorem{theorem}{Theorem}[section]

\newtheorem{lemma}[theorem]{Lemma}

\theoremstyle{definition}

\theoremstyle{remark}
\newtheorem{remark}[theorem]{Remark}

\theoremstyle{acknowledgement}
\newtheorem{acknowledgement}[theorem]{Acknowledgement}

\begin{document}
\title[Maximum likelihood drift estimation]
{Maximum likelihood drift estimation for the mixing of two  fractional Brownian motions}
\author{Y. Mishura}%



\maketitle
\begin{abstract}
We construct the maximum likelihood estimator (MLE) of the unknown drift parameter $\theta\in \R$  in the linear   model
$$X_t=\theta t+\sigma B^{H_1}(t)+B^{H_2}(t),\;t\in[0,T],$$
where $B^{H_1}$ and $B^{H_2}$ are two independent fractional Brownian motions with Hurst indices $\frac12<H_1<H_2<1.$ The formula for MLE is based on the solution of the integral equation with weak polar kernel.
\end{abstract}



\section{Introduction. The elements of stochastic calculus   for fBm}\label{intro}
Consider the continuous-time linear model $$X(t)=\theta t+\sigma B^{H_1}(t)+ B^{H_2}(t),\;t\in[0,T],$$
where $B^{H_1}$ and $B^{H_2}$ are two independent fractional Brownian motions with Hurst indices $\frac12<H_1<H_2<1,$ $\sigma>0.$ Given the sample path of $X$ on $[0,T]$ it is required to estimate the unknown drift parameter $\theta\in \R$. In the case when $H_1=\frac12$ the problem was  solved in \cite{CaiChigKlept} where the authors develop the basic tools for analysis of the mixed fBm  based on the
filtering theory of Gaussian processes. They considered the
linear regression setting and demonstrated  how the maximum likelihood estimator   can be
constructed and studied in the large sample asymptotic regime.

As the preliminaries,  we give some basic facts about stochastic calculus for fractional Brownian motion (for more details see \cite{BiHuOksZhang}, \cite{Decr}, \cite{Jost}, \cite{Mi}, \cite{Nu}). Parameter estimation for different models with long memory was studied, among others,   in \cite{betotu}, \cite{bishwal}, \cite{HuNu}, \cite{prakasa rao}.
Let $(\Omega, \mathfrak{F}, P)$ be a complete probability space  supporting all stochastic processes considered in what follows. Introduce $\set{B^H(t),t\ge 0}$,  an adapted  fractional Brownian motion with Hurst parameter $H\in(1/2,1)$, i.e. a centered Gaussian process with the covariance function
$$
\ex{B^H(t)B^H(s)}=\frac12\left(t^{2H}+s^{2H}-\abs{t-s}^{2H}\right).
$$
For a (possibly infinite) interval $[0,T]$, denote by  $L^2_H[0,T]$ the completion of the space of simple functions $f\colon [0,T]\to\R$ with respect to the scalar product
$$
\pair{f,g}_{H}^2:= \alpha_H\int_0^T\int_0^T f(t)g(s)\abs{t-s}^{2H-2}dsdt,
$$
where $\alpha_H = H(2H-1)$. (It is worth to mention that this completion contains not only classical functions, but also some  distributions.)
For a step function of the form
$$
f(t) = \sum_{k=0}^{n-1}a_k \ind{[t_k,t_{k+1})},
$$
where $\set{t_0<t_1<\dots<t_n}\subset [0,T]$, the integral $I^H(f)$ of $f$ with respect to $B^H$ is defined by
$$
I^H(f)=\int_{0}^T f(t) dB^H(t) = \sum_{k=0}^{n-1} a_k \left(B^H(t_{k+1})-B^H(t_k)\right).
$$
It can be verified that $I^H$ maps isomorphically the space of step functions on $[0,T]$ with the scalar product $\pair{\cdot,\cdot}_H$  into $L^2(\Omega)$, hence, $I^H$ can be extended to  $L^2_H[0,T]$.

Now define a square integrable kernel
$$
K_{H}(t,s) = \beta_H s^{1/2-H}\int_s^t (u-s)^{H-3/2}u^{H-1/2}du,
$$
where $\beta_H = \big(\frac{\alpha_H}{\mathrm{B}(H-1/2,2-2H)}\big)^\frac12,$ where $B(\cdot,\cdot)$ is beta function.
The map
$$
(K^*_H f)(s) = \int_{s}^{T} f(t) \partial_t K(t,s)dt =
\beta_H s^{1/2-H} \int_{s}^{T} f(t) t^{H-1/2}(t-s)^{H-3/2}dt
$$
is an isometry between the space of step functions and and can be extended to a Hilbert space isomorphism between $L^2_H[0,T]$ and
$L^2[0,T]$. This implies that the process
$$
W(t) = I^H\left((K^*_H)^{-1}\ind{[0,t]}\right)
$$
is a standard Wiener process on $[0,T]$, moreover,
for any $f\in L^2_H[0,T]$,
\begin{equation}\label{IH}
I^H(f) = \int_0^T (K^*_H f)(s)dW(s).
\end{equation}
In particular, putting in the last formula $f=\ind{[0,t]}$, one gets the following well-known representation of $B^H$:
$$
B^H(t)=\int_{0}^{t} K_H(t,s)dW(s).
$$
Finally, we define the so-called fundamental martingale, or Molchan martingale $M^H$, for $B^H$. In this order, introduce the kernel $l_H(t,s)=(t-s)^{1/2-H}s^{1/2-H}$ and consider square-integrable Gaussian martingale
$$M^H(t) = \int_0^t l_H(t,s)dB^H (s) = \gamma_H \int_0^t s^{1/2-H} dW(s),
$$
with
$\gamma_H = (2H(\frac32-H) \Gamma(3/2-H)^3\Gamma(H+\frac12)\Gamma(3-2H)^{-1})^{\frac12}$; the last equality is due to \eqref{IH}.

The paper is organized as follows. In Section 2 we reduce the main problem to the solution of the integral equation with the weak polar kernel and establish the existence-uniqueness result for this equation. Section 3 contains an auxiliary result concerning the existence and uniqueness of the solution of the corresponding integral equation of the 1st kind. We prove this fact directly, by constructing the unique solution.

\section{Main problem}

Now, let $\frac12<H_1<H_2<1,$ $\set{\widetilde{B}^{H_1}(t), B^{H_2}(t),\;t\ge 0}$, $i=1,2$, be two processes defined on the space $(\Omega, \mathfrak{F}, (\mathfrak{F})_t)$ and $P_\theta$ be a probability measure under which $\widetilde{B}^{H_1}$ and $B^{H_2}$ are independent,  $B^{H_2}$ is a   fractional Brownian motion  with Hurst parameter  $H_2$, and $\widetilde{B}^{H_1}$
is a fractional Brownian motion with Hurst parameter $H_1$  and with drift $\frac\theta\sigma$, i.e.,
 $$\sigma\widetilde{B}^{H_1}(t)=\theta t+\sigma{B}^{H_1}(t).$$

The probability measure $P_0$ corresponds to the case when $\theta=0$. Our main problem is the construction of maximum likelihood estimator for $\theta\in\R$ by the observations of the process $Z(t)=\theta t+\sigma B^{H_1}(t)+ B^{H_2}(t),\;t\in[0,T].$ However, the form of the process $Z$ (two fBm's with different Hurst indices) does not allow to construct the estimator immediately. To simplify the construction, we apply to $Z$ the linear transformation of the following form:
\begin{equation}\begin{gathered}\label{lin.trans}Y(t)=\int_0^tl_{H_1}(t,s)dZ(s)=\theta B\Big(\frac32-H_1,\frac32-H_1\Big)t^{2-2H_1}+\sigma M^{H_1}(t)\\+\int_0^t l_{H_1}(t,s)dB^{H_2}(s). \end{gathered}\end{equation}
this process is preferable since it involves Gaussian martingale $M^H$.
\begin{lemma}
The linear transformation \eqref{lin.trans} is correctly defined.
\end{lemma}
\begin{proof}

It is sufficient to establish the existence of the integral $\int_0^t l_{H_1}(t,s)dB^{H_2}(s)$ for any $t\in[0,T]$. But we have that for any $u,s\in[0,t]$ $$|u-s|^{2H_2-2}\leq t^{2H_2-2H_1}|u-s|^{2H_1-2},$$
therefore\begin{equation*}\begin{gathered}
||l_{H_1}(t,\cdot)||_{H_2}^2:= \alpha_{H_2}\int_0^t\int_0^t l_{H_1}(t,s)l_{H_1}(t,u)\abs{u-s}^{2{H_2}-2}dsdu\\\leq \alpha_{H_2}t^{2H_{2}-2H_1}\int_0^t\int_0^t l_{H_1}(t,s)l_{H_1}(t,u)\abs{u-s}^{2{H_1}-2}dsdu \\=\alpha_{H_2}t^{2H_{2}-2H_1}||l_{H_1}(t,\cdot)||_{H_1}^2=\alpha_{H_2}t^{2H_{2}-2H_1}E|M^{H_1}(t)|^2\\
=\frac{\alpha_{H_2}\gamma^2_{H_1}}{2-2H_1}t^{2H_{2}-4H_1+2}<\infty,
\end{gathered}\end{equation*}
whence the proof follows.
\end{proof}

As it was mentioned, process $Y$ is more convenient to deal with since it involves martingale with a drift. Furthermore, it follows from the next result that processes $Z$ and $Y$ are observed simultaneously, so, we can reduce the original problem to the equivalent problem of  the construction of maximum likelihood estimator of $\theta\in\R$ basing on  the linear transformation $Y$.
\begin{lemma} Processes $Z$ and $Y$ are observed simultaneously.
\end{lemma}
\begin{proof} Taking into account \eqref{lin.trans}, it  is enough to present $Z$ via $Y$. But it follows from \eqref{lin.trans}, from   Fubini theorem for  integrals w.r.t fBm (Theorem 2.6.5 \cite{Mi}), and from elementary integral transformations, that for any $t\in[0,T]$
\begin{equation*}\begin{gathered} \int_0^t(t-s)^{H_1-\frac32} \int_0^sl_{H_1}(s,u)dZ(u) ds=\int_0^tu^{\frac12-H_1}\int_u^t(t-s)^{H_1-\frac32}(s-u)^{\frac12-H_1}dsdZ(u)\\
=B\Big(H_1-\frac12,\frac32-H_1\Big)\int_0^tu^{\frac12-H_1}dZ(u)=\int_0^t(t-s)^{H_1-\frac32}Y(s)ds\\
=\Big(H_1-\frac12\Big)^{-1}\int_0^t (t-s)^{H_1-\frac12}dY(s),
\end{gathered}\end{equation*}
whence \begin{equation*}\begin{gathered} Z(t)=B\Big(H_1-\frac12,\frac32-H_1\Big)^{-1}\int_0^t\int_s^t (u-s)^{H_1-3/2}u^{H_1-1/2}du dY(s),
\end{gathered}\end{equation*}
and the proof follows.
\end{proof}
Denote for simplicity $\mathcal{B}_{H_1}:=B\Big(\frac32-H_1,\frac32-H_1\Big)$. Now the main problem can be formulated as follows.  Let $\frac12<H_1<H_2<1,$ $\{\widetilde{X}_1(t)=\widetilde{M}^{H_1}(t), X_2(t):=\int_0^t l_{H_1}(t,s)dB^{H_2}(s),\;t\ge 0\}$, $i=1,2$, be two processes defined on the space $(\Omega, \mathfrak{F})$ and $P_\theta$ be a probability measure under which $\widetilde{X}_1$ and $X_2$ are independent,  $B^{H_2}$ is a   fractional Brownian motion  with Hurst parameter  $H_2$, and $\widetilde{X}_1$
is a martingale with square characteristics $\langle  \widetilde{X}_1\rangle(t)=\frac{\gamma^2_{H_1}}{2-2H_1}t^{2-2H_{1}}$ and with drift $\frac{\theta B}{\sigma }t^{2-2H_1}$, i.e.,
 $$\widetilde{X}_1(t)=\widetilde{M}^{H_1}(t)=\frac{\theta \mathcal{B}_{H_1}}\sigma t^{2-2H_1}+{M}^{H_1}(t).$$

Also, denote  ${X}_1(t)={M}^{H_1}(t)$. Our main problem is the construction of maximum likelihood estimator for $\theta\in\R$ by the observations of the process $$Y(t)=\theta \mathcal{B}_{H_1}t^{2-2H_1}+\sigma X_1(t)+X_2(t). $$ Note  that under measure $P_\theta$ the process
 $$\widetilde{W}(t):=W(t)+ \frac{\theta(2-2H_1)\mathcal{B}_{H_1}}{\sigma\gamma_{H_1}\Big(\frac32-H_1\Big)}t^{\frac32-H_1} $$
 is a Wiener process with drift. Denote $\delta_{H_1}=\frac{(2-2H_1)\mathcal{B}_{H_1}}{\sigma\gamma_{H_1}}$.

By Girsanov's theorem and independence of $X_1$ and $X_2$,
\begin{equation*}\begin{gathered}\frac{dP_\theta}{dP_0}=\exp\Big\{\theta\delta_{H_1}\int_0^Ts^{\frac12- H_1}d\widetilde{W}(s)
- \frac{\theta^2\delta_{H_1}^2}{4(1- H_1)}T^{2-2H_1}\Big\}\\=\exp\Big\{\theta\delta_{H_1}\widetilde{X}_1(T)
-\frac{\theta^2\delta_{H_1}^2}{4(1- H_1)}T^{2-2H_1}\Big\}.\end{gathered}\end{equation*}

For technical simplicity, we put $\sigma=1$ in what follows. Note that, similarly to the linear combination of Wiener process and fBm, considered in \cite{CaiChigKlept}, this derivative is not the likelihood for the problem at hand, because it is not measurable with respect to the observed $\sigma$-algebra
 $$\mathfrak{F}^Y_T:=\sigma\{Y(t),t\in[0,T]\}=\mathfrak{F}^X_T:=\sigma\{X(t),t\in[0,T]\},$$
 where $X(t)=X_1(t)+X_2(t).$

 We shall proceed as in \cite{CaiChigKlept}: let $\mu_\theta$ be the probability measure induced by $Y$ on the space of continuous functions with the supremum topology under probability $P_\theta$. Then for any measurable set $A$
  $\mu_\theta(A)=\int_A\Phi(x)\mu_0dx,$
where $\Phi(x)$ is such measurable functional that $\Phi(X)=E_0\Big(\frac{dP_\theta}{dP_0}\Big|\mathfrak{F}^X_T\Big)$. The latter means that $\mu_\theta\ll\mu_0$ for any $\theta\in\R$.  Taking into account that  $\widetilde{X}_1=X_1$ under $P_0$ and the fact that the vector process $(X_1, X)$ is Gaussian, we get that the  corresponding likelihood function is given by
\begin{equation}\begin{gathered}\label{eq.prime2.1}L_T(X,\theta)=E_0\Big(\frac{dP_\theta}{dP_0}\Big|\mathfrak{F}^X_T\Big)
=E_0\Big(\exp\Big\{\theta\delta_{H_1}{X}_1(T)
-\frac{\theta^2\delta_{H_1}^2}{4(1- H_1)}T^{2-2H_1}\Big\}\Big|\mathfrak{F}^X_T\Big)\\
=\exp\Big\{\theta\delta_{H_1}E_0({X}_1(T)|\mathfrak{F}^X_T)+\frac{\theta^2\delta_{H_1}^2}{2}
\Big(V(T)-\frac{T^{2-2H_1}}{2-2H_1}\Big)\Big\},
\end{gathered}\end{equation}
where $V(t)=E_0(X_1(t)-E_0(X_1(t)|\mathfrak{F}^X_t))^2,\;t\in[0,T].$


Thus, we arrive at the following problem: to find the projection $P_X X_1(T)$ of $X_1(T)$ onto $\set{X(t) = X_1(t)+X_2(t),t\in[0,T]}$.
We recall from Section~\ref{intro} that
$$
W_i(t) =  \int_0^t \left((K^*_{H_i})^{-1}\ind{[0,t]}\right) dB^{H_i}(s),\ i=1,2,
$$
are standard Wiener processes, which are obviously independent. Also from Section~\ref{intro} we have
\begin{equation}\label{X1}
X_1(t) = \gamma_{H_1} \int_0^t s^{1/2-H_1} dW_1(s),\;
B^{H_2}(t) = \int_0^t K_{H_2}(t,s)dW_2(s).
\end{equation}
Then, using \eqref{IH}, we can write
$$
X_2(t) = \int_0^t K_{H_1,H_2}(t,s) dW_2(s),
$$
where
\begin{equation}\label{kernelasitis}
K_{H_1,H_2}(t,s) = \beta_{H_2} s^{1/2-H_2}\int_{s}^{t}(t-u)^{1/2-H_1}u^{H_2-H_1}(u-s)^{H_2-3/2}du.
\end{equation}

Similarly to \eqref{IH}, we have for $f\in L^2_{H_2}[0,T]$
\begin{equation}\label{intX2W2}
\int_0^T f(s)dX_2(s) = \int_{0}^{T}(K^*_{H_1,H_2}f)(s)dW_2(s),
\end{equation}
where
$$
(K^*_{H_1,H_2}f)(s) = \int_s^T f(t)\partial_t K_{H_1,H_2}(t,s)dt.
$$

 The projection of $X_1(T)$ onto $\set{X(t),t\in[0,T]}$ is a centered $X$-measurable Gaussian random variable, therefore, it has a form
$$
P_X X_1(T) = \int_0^T  h_T(t)dX(t)
$$
with $h_T \in L^2_{H_1}[0,T]$. Note that $h_T$ still can be a distribution. 
This projection for all $u\in[0,T]$ must satisfy
\begin{equation}\label{eq2.3}
\ex{X(u) P_X X_1(T)} = \ex{X(u)X_1(T)}.
\end{equation}
Using \eqref{eq2.3}  together with independency of $X_1$ and $X_2$, we arrive at

\begin{equation}\label{eq2.4}
\ex{X_1(u)\int_{0}^{T}h_T(t) dX_1(t) + X_2(u)\int_{0}^{T}h_T(t) dX_2(t)} = \ex{X_1(u)X_1(T)} = \varepsilon_{H_1} u^{2-2H_1},
\end{equation}
where $\varepsilon_{H} = \gamma_H^2/(2-2H)$. 

From \eqref{X1} -- \eqref{eq2.4}  we get
\begin{equation}\label{main1}
\varepsilon_{H_1} u^{2-2H_1} = \gamma_{H_1}^2 \int_0^u h_T(s)s^{1-2H_1} ds + \int_0^T h_T(s)r_{H_1,H_2}(s,u)ds,
\end{equation}
where
$$
r_{H_1,H_2}(s,u)=\int_0^{s\wedge u}\partial_s K_{H_1,H_2}(s,v)K_{H_1,H_2}(u,v)dv.
$$
This kernel can be written alternatively as $r_{H_1,H_2}(t,s) = \partial_t R_{H_1,H_2}(t,s)$, where
\begin{gather*}
R_{H_1,H_2}(t,s)=\int_0^{t\wedge s} K_{H_1,H_2}(t,u)K_{H_1,H_2}(s,u)du = \ex{X_2(t)X_2(s)} \\ =
\alpha_{H_2}\int_{0}^{t}\int_{0}^{s} (t-u)^{1/2-H_1}u^{1/2-H_1}(s-v)^{1/2-H_1}
v^{1/2-H_1}|u-v|^{2H_2-2}dv\,du.
\end{gather*}
Differentiating \eqref{main1} with respect to $u$, we arrive to
\begin{equation}\label{eq2.7}
\gamma_{H_1}^2 u^{1-2H_1}  = \gamma_{H_1}^2 h_T(u)u^{1-2H_1}
+ \int_0^T h_T(s) k(s,u)ds,
\end{equation}
where
\begin{equation}\label{eq2.8}
k(s,u) = \partial_u r_{H_1,H_2}(s,u) = \int_0^{s\wedge u}\partial_s K_{H_1,H_2}(s,v)\partial_u K_{H_1,H_2}(u,v)dv.
\end{equation}
\begin{theorem}\label{theo2.3} Let $H_2-H_1>\frac14.$
Then  there exists a sequence $T_n\rightarrow \infty $ such that integral equation \eqref{eq2.7} has unique solution $h_{T_n}$ on any interval $[0,T_n]$ and $h_{T_n}(\cdot)\cdot^{\frac12- H_1}\in L_2[0,_n]$.
\end{theorem}
\begin{proof} We denote $C_{H_1,H_2}$ constants which values are not so important;  their values can change from line to line.
At first, we can apply the changing of variables $u=s+(t-s)z$ to \eqref{kernelasitis} and transform the kernel $K_{H_1,H_2}(t,s)$ from \eqref{kernelasitis} to the following form:
\begin{equation}\label{eq2.9}\begin{gathered}
K_{H_1,H_2}(t,s)=\beta_{H_2}s^{\frac12-H_2}(t-s)^{H_2-H_1}\int_0^1(1-z)^{\frac12-H_1}(s+(t-s)z)^{H_2-H_1}z^{H_2-\frac32}dz.
 \end{gathered}
\end{equation}
Then we can differentiate \eqref{eq2.9} and  after inverse changing of variables we get that
 \begin{equation}\label{eq2.10}\begin{gathered} \partial_t K_{H_1,H_2}(t,s)
 =(H_2-H_1)\bigg(\frac{K_{H_1,H_2}(t,s)}{t-s}\\
 +\beta_{H_2}s^{\frac12-H_2}\frac{1}{t-s}\int_s^t(t-r)^{\frac12-H_1}r^{H_2-H_1-1}(r-s)^{H_2-\frac12}dr\bigg).
 \end{gathered}
\end{equation}
Further, we have the following bound  for kernel $K_{H_1,H_2}(t,s)$ on the interval $[0,T]$ :
\begin{equation}\label{eq2.11}\begin{gathered}
0\leq K_{H_1,H_2}(t,s)\leq \beta_{H_2}B\Big(\frac32-H_1, H_2-\frac12\Big)t^{H_2-H_1}s^{\frac12-H_2}(t-s)^{H_2-H_1},
\end{gathered}
\end{equation} and it follows from \eqref{eq2.10} and \eqref{eq2.11} that
\begin{equation}\label{eq2.12}\begin{gathered} 0\leq \partial_t K_{H_1,H_2}(t,s)\leq \beta_{H_2}s^{\frac12-H_2}\Big(B\Big(\frac32-H_1, H_2-\frac12\Big)t^{H_2-H_1}(t-s)^{H_2-H_1-1}\\+B\Big(\frac32-H_1, H_2+\frac12\Big)t^{H_2-H_1-1}(t-s)^{H_2-H_1}\Big)\leq C_{H_1,H_2}s^{\frac12-H_2}t^{H_2-H_1}(t-s)^{H_2-H_1-1}.
\end{gathered}
\end{equation}
Now we can substitute the bound from \eqref{eq2.12} into \eqref{eq2.8} and get that
\begin{equation}\label{eq2.14}\begin{gathered}0\leq k(s,u)\leq C_{H_1,H_2}u^{H_2-H_1}s^{H_2-H_1}\int_0^{s\wedge u}v^{1-2H_2}(u-v)^{H_2-H_1-1}(s-v)^{H_2-H_1-1}dv.
\end{gathered}
\end{equation}
Let, for example, $s<u$. Note that $$(u-v)^{H_2-H_1-1}=(u-v)^{H_2+H_1-2}(u-v)^{1-2H_1}\leq (u-v)^{H_2+H_1-2}(u-s)^{1-2H_1}.$$
Then it follows from \eqref{eq2.14} that
\begin{equation}\label{eq2.15}\begin{gathered}0\leq k(s,u)
\leq C_{H_1,H_2}u^{H_2-H_1}s^{H_2-H_1}(u-s)^{1-2H_1}\\\times\int_0^{s}v^{1-2H_2}(u-v)^{H_2+H_1-2}(s-v)^{H_2-H_1-1}dv.
\end{gathered}
\end{equation}
In order to bound the integral in the right-hand side of \eqref{eq2.15}, we apply the first statement from lemma 2.2 \cite{NVV}, according to which for $\mu,\nu>0$, $c>1$ $$\int_0^1t^{\mu-1}(1-t)^{\nu-1}(c-t)^{-\mu-\nu}dt=B(\mu,\nu)c^{-\nu}(c-1)^{-\mu}.$$
Therefore, with $\mu=2-2H_2,\nu=H_2-H_1$ and $c=\frac us$
\begin{equation}\begin{gathered}\label{eq2.16}
\int_0^{s}v^{1-2H_2}(u-v)^{H_2+H_1-2}(s-v)^{H_2-H_1-1}dv\\=s^{-1}\int_0^{1}v^{1-2H_2}\Big(\frac us-v\Big)^{H_2+H_1-2}(1-v)^{H_2-H_1-1}dv\\=B(2-2H_2,H_2-H_1)s^{-1}\Big(\frac us\Big)^{H_1-H_2}\Big(\frac us-1\Big)^{2H_2-2}\\=C_{H_1,H_2}s^{1-H_2-H_1}u^{H_1-H_2}(u-s)^{2H_2-2},
\end{gathered}\end{equation}
and it follows from \eqref{eq2.15} and \eqref{eq2.16} that for $s<u$  \begin{equation}\label{eq2.17}\begin{gathered}0\leq k(s,u)
\leq C_{H_1,H_2}s^{1-2H_1}(u-s)^{2H_2-2H_1-1}.
\end{gathered}
\end{equation}
Evidently, for $u<s$
\begin{equation}\label{eq2.18}\begin{gathered}0\leq k(s,u)
\leq C_{H_1,H_2}u^{1-2H_1}(s-u)^{2H_2-2H_1-1}.
\end{gathered}
\end{equation}
Now we rewrite equation \eqref{eq2.7} in the equivalent  form
\begin{equation}\label{eq2.19}
\gamma_{H_1}^2u^{\frac12-H_1}  = \gamma_{H_1}^2 h (u)u^{\frac12-H_1}
+ \int_0^T h (s)s^{\frac12-H_1}s^{H_1-\frac12}u^{H_1-\frac12}k(s,u)ds,
\end{equation}
or
\begin{equation}\label{eq2.19'}
\gamma_{H_1}^2u^{\frac12-H_1}  = \gamma_{H_1}^2 \widetilde{h}_T(u)
+ \int_0^T \widetilde{h}_T (s) k_1(s,u)ds,
\end{equation}
where  $k_1(s,u)=s^{H_1-\frac12}u^{H_1-\frac12} k(s,u),$ $\widetilde{h}_T (u)=h (u)u^{\frac12-H_1}$ and it follows from \eqref{eq2.17} and \eqref{eq2.18} that for $s<u$
$$k_1(s,u)\leq C_{H_1,H_2}u^{H_1-\frac12}s^{\frac12-H_1}(u-s)^{2H_2-2H_1-1}$$
 and for $u<s$
$$k_1(s,u)\leq C_{H_1,H_2}s^{ H_1-\frac12}u^{\frac12 - H_1}(s-u)^{2H_2-2H_1-1}.$$
Therefore, taking into account that for  $H_2-H_1>\frac14$ we have that $4H_2-4H_1-2>-1$, it is possible to  bound $L_2[0,T]^2$ - norm of the kernel:
\begin{equation}\label{eq2.20}\begin{gathered}
\|k_1\|_{L_2[0,T]^2}=\int_0^T\int_0^Tk_1^2(s,u)dsdu=\int_0^T\int_0^uk_1^2(s,u)dsdu+\int_0^T\int_u^Tk_1^2(s,u)dsdu\\\leq
 C_{H_1,H_2}(\int_0^T\int_0^uu^{2H_1-1}s^{1- 2H_1}(u-s)^{4H_2-4H_1-2}dsdu\\+\int_0^T\int_u^Ts^{2H_1 -1}u^{1-2H_1}(s-u)^{4H_2-4H_1-2}dsdu)\\\leq C_{H_1,H_2}\Big(\int_0^Tu^{4H_2-4H_1-1}du+T^{2H_1-1}\int_0^Tu^{1-2H_1}(T-u)^{4H_2-4H_1-1}du\Big)\\\leq C_{H_1,H_2}T^{4H_2-4H_1}<\infty. \end{gathered}
\end{equation}
It means that the integral operator $K_Tf(u)=\int_0^Tk_1 (s,u)f(u)du$ is compact linear self-adjoint operator from  $L_2[0,T]$ into $L_2[0,T]$ and Fredholm alternative can be applied to equation \eqref{eq2.19'}. To avoid the question concerning eigenvalues and eigenfunctions, we produce the following trick.

It is very easy to see that for any $a>0$ \begin{equation}\begin{gathered}\label{equ2.26'}
K(ta,sa)=K(t,s)a^{\frac12+H_2-2H_1},\;\partial_t K_{H_1,H_2}(ta,sa)=\partial_tK(t,s)a^{-\frac12+H_2-2H_1},\\
k(ta,sa)=a^{2H_2-4H_1}k(t,s),\end{gathered}\end{equation}
whence $$k_1(ta,sa)=k_1(t,s)a^{2H_2-2H_1-1}.$$
Therefore   we can put in equation \eqref{eq2.19'} $s=s' {T} , u=u' {T} $ and $\widehat{h}_{T}(z)=\widetilde{h}_T({T}z)$, and equation  \eqref{eq2.19'} will be reduced to the equivalent form (we omit superscripts)
\begin{equation}\label{eq2.21}\begin{gathered}
(u T )^{\frac12-H_1}=  \widehat{h}_{T}(u)
+   T^{2H_2-2H_1}\gamma_{H_1}^{-2}\int_0^{1} \widehat{h}_{T}(s) k_1(s,u)ds= \widehat{h}_{T}(u)\\ +\lambda\int_0^{1} \widehat{h}_{T}(s) k_1(s,u)ds,
\end{gathered}\end{equation}
with $\lambda= T^{2H_2-2H_1}\gamma_{H_1}^{-2}.$ Since operator $K_1$ is compact linear self-adjoint operator from  $L_2[0,T]$ into $L_2[0,T]$, as it was mentioned above, it has no more than countable number of eigenvalues any of them are real numbers, and with only one possible condensation point $0$. Taking the sequence $T_n\rightarrow \infty$ in such a way that $$\lambda_n= T_n^{2H_1-2H_2}\gamma_{H_1}^{2} $$ will be not an eigenvalue, we get
  that equation \eqref{eq2.21}  with $T_n$ as upper bound of integration has unique solution whence the proof follows.
\end{proof}

Now we establish the form of maximum likelihood estimate.
\begin{theorem} Let  $H_2-H_1>\frac14.$
Then the likelihood function has a form \begin{equation}\begin{gathered}\label{eq2.25}L_T(X,\theta)=\exp\{\theta\delta_{H_1}N(T)-\frac12\theta^2\delta^2_{H_1}\langle N\rangle(T)\},\end{gathered}
\end{equation}
and maximum likelihood estimate has a form \begin{equation}\begin{gathered}\label{eq2.26}\widehat{\theta}(T)=\frac{N(T)}{\delta_{H_1}\langle N\rangle(T)},\end{gathered}
\end{equation}
where $N(t)=E_0(X_1(t))|\mathfrak{F}^X_t)$  is a square integrable Gaussian $\mathfrak{F}^X_t$-martingale, $N(T)=\int_0^Th_T(t)dX(t)$  with $h_T(t)t^{\frac12-H_1}\in L_2[0,T]$, $h_T(t)$ be a unique solution to \eqref{eq2.7} and $\langle N\rangle(T)=\gamma^2_{H_1}\int_0^T h_T(t)t^{1-2H_1}dt.$
\end{theorem}
\begin{proof} We start with \eqref{eq.prime2.1}. Consider Gaussian process $N(t)=E_0(X_1(t)|\mathfrak{F}^X_t)$. Since $X_1(t)$ is $\mathfrak{F}_t$-martingale and $\mathfrak{F}^X_t\subset\mathfrak{F}_t$, the process $N$ is a $\mathfrak{F}^X_t$-martingale with respect to probability measure $P_0$. Furthermore, we can present $V(t)$ as $V(t)=E_0(X_1^2(t))|\mathfrak{F}^X_t)-N^2(t)$. Note that $X_1^2(t)-\frac{t^{2-2H}}{2-2H}$ is $\mathfrak{F}_t$-martingale. Therefore,
 \begin{equation}\begin{gathered} E_0\left(N^2(t)-\left(\frac{t^{2-2H}}{2-2H}-V(t)\right)\Big|\mathfrak{F}^X_s\right)
 =E_0\left( E_0(X_1^2(t) |\mathfrak{F}^X_t)-\frac{t^{2-2H}}{2-2H}\Big|\mathfrak{F}^X_s\right)
 \\=E_0\left(X_1^2(t)-\frac{t^{2-2H}}{2-2H}\Big|\mathfrak{F}^X_s\right)
 =E_0(X_1^2(s)|\mathfrak{F}^X_s)-\frac{s^{2-2H}}{2-2H}\\=N^2(s)-\left(\frac{s^{2-2H}}{2-2H}-V(s)\right),
\end{gathered}
\end{equation}
and this means that the quadratic variation of the martingale $N$ equals $\langle N\rangle(t)=t^{2-2H}-V(t)$, and the likelihood ratio is reduced to

\begin{equation}\begin{gathered}L_T(X,\theta)=\exp\Big\{\theta\delta_{H_1}N(T)-\frac12{\theta^2\delta_{H_1}^2}\langle N\rangle(T)\Big\},\end{gathered}
\end{equation}
so, we get \eqref{eq2.25} and \eqref{eq2.26}. Now, taking \eqref{eq2.7} into account, we get that
\begin{equation*}\begin{gathered}\langle N\rangle(T)=E_0(N^2(T))=E_0\Big(\int_0^Th_T(t)dX(t)\Big)^2=E_0\Big(\int_0^Th_T(t)d(X_1(t)+X_2(t)\Big)^2
\\=E_0\Big(\int_0^Th_T(t)d X_1(t)\Big)^2+E_0\Big(\int_0^Th_T(t)d X_2(t)\Big)^2=\gamma^2_{H_1}\int_0^Th_T^2(t)t^{1-2H_1}dt
\\+\int_0^T\int_t^Th_T(u)\partial_u K_{H_1,H_2}(u,t)du\int_t^Th_T(s)\partial_s K_{H_1,H_2}(s,t)dsdt\\
=\gamma^2_{H_1}\int_0^Th_T^2(t)t^{1-2H_1}dt+\int_0^Th_T(u)h_T(s)\int_0^{s\wedge u}\partial_s
 K_{H_1,H_2}(s,t)\partial_u K_{H_1,H_2}(u,t)dtdsdu\\=\gamma^2_{H_1}\int_0^Th_T(t)t^{1-2H_1}dt, \end{gathered}
\end{equation*}
whence the proof follows.
\end{proof}
In what follows, saying ``$T\rightarrow \infty$'' we have in mind that the corresponding property holds for any sequence $T_n\rightarrow \infty$ that has only finite common points with the sequence of eigenvalues of operator $K_1$. Proof of the following result repeats the proof of the corresponding statements from \cite{CaiChigKlept} so is omitted.
\begin{theorem}
The estimator $\widehat{\theta}_T$ is unbiased and the corresponding estimation error is normal
$$\widehat{\theta}_T-\theta\sim N\Big(0, \frac{1}{\int_0^Th_T(s)s^{1-2H_1}ds}\Big).$$
\end{theorem}

Now we establish the asymptotic behavior of the estimator.
\begin{theorem} Let $H_2-H_1>\frac14$.
Estimator $\widehat{\theta}_T$ is strongly consistent and $$\lim_{T\rightarrow \infty}T^{2-2H_2}E_\theta(\widehat{\theta}_T-\theta)^2=\frac{1}{\int_0^1h_0(u)u^{\frac12-H_1}du}.$$
\end{theorem}
\begin{proof} At first we rewrite equation \eqref{eq2.19'} in the equivalent form, changing $u=u'T, s=s'T$ and omitting superscripts:
\begin{equation}\label{eq2.7'}
\gamma_{H_1}^2 u^{\frac{1}{2}- H_1}T^{\frac{1}{2}- H_1}  = \gamma_{H_1}^2 \widetilde{h}_T(uT)
+ T^{2H_2-2H_1}\int_0^1 \widetilde{h}_T(sT) k_1(s,u)ds,
\end{equation}
or
\begin{equation}\label{eq2.33}
\gamma_{H_1}^2 u^{\frac{1}{2}- H_1}  = \gamma_{H_1}^2 \widetilde{h}_T(uT)T^{H_1-\frac{1}{2}}
+ T^{2H_2-2H_1}\int_0^1 \widetilde{h}_T(sT)T^{H_1-\frac{1}{2}} k_1(s,u)ds,
\end{equation}
 Denote $\mu=T^{2H_2-2H_1}$. Let $h_\mu(u)=\mu \widetilde{h}_T(uT)T^{H_1-\frac{1}{2}}$. Then, taking into account \eqref{equ2.26'}, equation \eqref{eq2.7'} can be rewritten as
\begin{equation}\label{eq2.8'}
\gamma_{H_1}^2u^{\frac{1}{2}- H_1}   = \frac{1}{\mu}\gamma_{H_1}^2 h_\mu(u)
+ \int_0^1 h_\mu(s) k_1(s,u)ds.
\end{equation}
Note that $$\langle N\rangle(T)=\int_0^Th_T(s)s^{1-2H_1}ds=\int_0^T\widetilde{h}_T(s)s^{\frac12- H_1}ds=T^{2-2H_2}\int_0^1h_\mu(u)u^{\frac12- H_1}du.$$
Define the operator $K$

$$(Kf)(u)=\int_0^1f(s)k_1(s,u)ds$$
and the scalar product $\langle f,g\rangle=\int_0^1 f(s)g(s)ds$, $f,g\in L_2[0,1].$
Then equation \eqref{eq2.8'} can be rewritten as \begin{equation}\label{eq2.32}
\gamma_{H_1}^2u^{\frac{1}{2}- H_1}   = \frac{1}{\mu}\gamma_{H_1}^2 h_\mu(u)
+ Kh_\mu(u).
\end{equation}
Note that \begin{equation}\begin{gathered}
\label{eq2.37}\langle Kf,f\rangle=\int_0^1 (Kf)(t)f(t)dt=\int_0^1(\int_0^1f(s)k_1(t,s)ds) f(t)dt\\=\int_0^1\int_0^1 f(t)t^{H_1-1/2}f(s)s^{H_1-1/2}\int_0^{s\wedge t}\partial_sK_{H_1,H_2}(s,v)\partial_tK_{H_1,H_2}(t,v)dv ds  dt\\=\int_0^1dv\int_v^1\partial_sK_{H_1,H_2}(s,v)f(s)s^{H_1-1/2}ds\int_v^1\partial_tK_{H_1,H_2}(t,v)f(t)t^{H_1-1/2} dt\geq 0.\end{gathered}\end{equation} Introduce corresponding the first type auxiliary integral equation
\begin{equation}\label{eq2.33}
\gamma_{H_1}^2 u^{\frac{1}{2}- H_1}  =  (Kh)(u).
\end{equation}
It follows from Lemma \ref{3.1} that \eqref{eq2.33} has the unique solution, say, $h_0$, obviously, not depending on $\mu$. The function $\delta_\mu=h_\mu-h_0$ satisfies two equations $K\delta_\mu+\frac{1}{\mu}\gamma_{H_1}^2 h_\mu =0 $ and $K\delta_\mu+\frac{1}{\mu}\gamma_{H_1}^2 \delta_\mu=-\frac{h_0}{\mu}.$
Multiplying the 2nd equation by  $\delta_\mu$ and integrating, we get \begin{equation}\label{eq2.38}\langle K\delta_\mu,\delta_\mu\rangle+\frac{1}{\mu}\gamma_{H_1}^2 \|\delta_\mu\|^2=\frac{1}{\mu}|\langle {h_0},\delta_\mu\rangle|,\end{equation}
and it follows from \eqref{eq2.38} and \eqref{eq2.37} that $\gamma_{H_1}^2\|\delta_\mu\|^2\leq |\langle {h_0},\delta_\mu\rangle|\leq \|h_0\|\|\delta_\mu\|,$ which implies that
$\|\delta_\mu\|\leq \|h_0\|.$ Multiplying the 1st equation by $h_0$ and integrating we get
$$\langle K\delta_\mu, h_0\rangle+\frac{1}{\mu}\gamma_{H_1}^2 \langle h_\mu,h_0\rangle =0.$$

Note that inequality $\|\delta_\mu\|\leq \|h_0\|$ implies that $$|\langle h_\mu,h_0\rangle|\leq |\langle \delta_\mu,h_0\rangle|+\|h_0\|^2\leq 2\|h_0\|^2<\infty,$$
and hence \begin{equation} \gamma_{H_1}^2|\langle \delta_\mu, u^{\frac12-H_1}\rangle|=|\langle\delta_\mu,Kh_0\rangle|=|\langle K\delta_\mu,h_0\rangle|=\frac{1}{\mu}\gamma_{H_1}^2 |\langle h_\mu,h_0\rangle|\rightarrow 0
\end{equation}
as $T\rightarrow \infty.$ It means that $\lim_{T\rightarrow \infty}\int_0^1h_\mu(u)u^{\frac12-H_1}du=\int_0^1h_0(u)u^{\frac12-H_1}du.$
Therefore $$T^{2-2H_2}E_\theta(\widehat{\theta}_T-\theta)^2= \frac{1}{\int_0^1h_\mu(u)u^{\frac12-H_1}du}\rightarrow \frac{1}{\int_0^1h_0(u)u^{\frac12-H_1}du},$$
whence the proof follows.
\end{proof}
\begin{remark}
In outline,  our method of proof follows the method of the corresponding result from \cite{CaiChigKlept}, however Lemma \ref{3.1} is specific to our case.
\end{remark}
\section{Appendix}
We recall some notions from fractional calculus. For the details see \cite{SamkoKM}. Fractional integrals are defined as $$(I_{a+}^\alpha f)(x)=\frac{1}{\Gamma(\alpha)}\int_a^xf(t)(x-t)^{\alpha-1}dt$$
and $$(I_{b-}^\alpha f)(x)=\frac{1}{\Gamma(\alpha)}\int_x^bf(t)(t-x)^{\alpha-1}dt,$$
while fractional derivatives are defined as $$(\mathcal{D}_{a+}^\alpha f)(x)=\frac{1}{\Gamma(1-\alpha)}\frac{d}{dx}\int_a^xf(t)(x-t)^{-\alpha}dt$$
and $$(\mathcal{D}_{b-}^\alpha f)(x)=-\frac{1}{\Gamma(1-\alpha)}\frac{d}{dx}\int_x^bf(t)(t-x)^{-\alpha}dt.$$
Fractional differentiation and integration are inverse operators on the appropriate functional classes. Also, we shall use   the following integration by parts formula for fractional derivatives,
$$\int_a^b(\mathcal{D}_{a+}^\alpha f)(x)g(x)dx=\int_a^b f(x)(\mathcal{D}_{b-}^\alpha g)(x)dx.$$
\begin{lemma}\label{3.1}
For any constant $C>0$ integral equation
\begin{equation}\label{eqA.1}
u^{1/2-H_1}=C(Kh)(u), \;u\in(0,1]
\end{equation}
of the 1st kind has the unique solution.
\end{lemma}
\begin{proof}
We can present equation \eqref{eqA.1} in  equivalent form  $$u^{1/2-H_1}=C\int_0^1h(s)k_1(s,u)ds), \;u\in(0,1],$$
or $$u^{1-2H_1}=C\int_0^1\breve{h}(s)\int_0^{s\wedge u}\partial_sK_{H_1,H_2}(s,v)\partial_uK_{H_1,H_2}(u,v)dv ds, \;u\in(0,1],$$
where $\breve{h}(s)={h}(s)s^{1/2-H_1}$, or, at last,
\begin{equation}\label{eqA.2}u^{1-2H_1}=C\int_0^u\bigg(\int_v^1\breve{h}(s)
\partial_sK_{H_1,H_2}(s,v)ds\bigg)\partial_uK_{H_1,H_2}(u,v)dv.\end{equation}
Now, taking into account the transition from equation \eqref{main1} to \eqref{eq2.7} with the help of differentiation, we can perform the inverse operation and get from \eqref{eqA.2} the following equivalent equation
\begin{equation}\label{eqA.3}u^{2-2H_1}
=C(2-2H_1)\int_0^uK_{H_1,H_2}(u,v)\bigg(\int_v^1\breve{h}(s)\partial_sK_{H_1,H_2}(s,v)ds\bigg)dv, u\in[0,1].\end{equation}
The right-hand side of equation \eqref{eqA.3} can be rewritten as
$$C(2-2H_1)\int_0^uK_{H_1,H_2}(u,v)q(v)dv,$$
where $q(v)=\int_v^1\breve{h}(s)\partial_sK_{H_1,H_2}(s,v)ds.$
At first, solve the equation $$u^{2-2H_1}=C_1\int_0^uK_{H_1,H_2}(u,v)q(v)dv,$$
with $C_1=C(2-2H_1)$. Taking into account \eqref{kernelasitis}, the latter equation can be rewritten in   equivalent form
$$u^{2-2H_1}=C_1\beta_{H_2}\int_0^uv^{1/2-H_2}\int_v^u (u-z)^{1/2-H_1}z^{H_2-H_1}(z-v)^{H_2-3/2}dz q(v)dv,$$
or $$u^{2-2H_1}=C_1\beta_{H_2}\int_0^uz^{H_2-H_1}(u-z)^{1/2-H_1}\int_0^zv^{1/2-H_2}(z-v)^{H_2-3/2}q(v)dv dz,$$
or, at last,
$$u^{2-2H_1}=C_2(I_{0+}^{3/2-H_1}p)(u),$$
where $C_2=C_1\beta_{H_2}\Gamma(3/2-H_2)$ and \begin{equation}\label{A.4}
p(z)=z^{H_2-H_1}\int_0^zv^{1/2-H_2}(z-v)^{H_2-3/2}q(v)dv. \end{equation}
It means that \begin{equation}\begin{gathered}\label{A.5}
p(u)=C_2^{-1}(\mathcal{D}^{3/2-H_1}_{0+}(\cdot^{2-2H_1}))(u)\\=(C_2\Gamma(H_1-1/2))^{-1}
\Big(\int_0^u(u-t)^{H_1-3/2}t^{2-2H_1}dt\Big)'_u=C_3u^{1/2-H_1},\end{gathered}\end{equation}
where $C_3=\frac{(3/2-H_1)B(H_1-1/2,3-2H_1)}{C_2\Gamma(H_1-1/2)}$. Furthermore, comparing \eqref{A.4} and \eqref{A.5},
we get that $$C_3z^{1/2-H_2}=\int_0^z v^{1/2-H_2}(z-v)^{H_2-3/2}q(v)dv=\Gamma(H_2-1/2)(I_{0+}^{H_2-1/2}(\cdot^{1/2-H_2}q))(z),$$
whence \begin{equation*}\begin{gathered}v^{1/2-H_2}q(v)=C_3(\Gamma(H_2-1/2))^{-1}
(\mathcal{D}^{H_2-1/2}_{0+}\cdot^{1/2-H_2})(v)\\
=C_3(\Gamma(H_2-1/2)\Gamma(3/2-H_2))^{-1}\Big(\int_0^v(v-t)^{1/2-H_2}t^{1/2-H_2}dt\Big)'_v\\=
C_4v^{1-2H_2},\end{gathered}\end{equation*}
where $C_4=C_3(2-2H_2)(\Gamma(H_2-1/2)\Gamma(3/2-H_2))^{-1}$.
Obviously, $q(v)=C_4v^{1/2-H_2}$,   and we arrive at the equation \begin{equation}\label{A.6}C_4v^{1/2-H_2}=\int_v^1\breve{h}(s)\partial_sK_{H_1,H_2}(s,v)ds.\end{equation}
Note that $$\partial_sK_{H_1,H_2}(s,v)=\beta_{H_2}\Gamma(3/2-H_1)v^{1/2-H_2}\Big(\mathcal{D}^{H_1-1/2}_{v+}
\Big(\cdot^{H_2-H_1}(\cdot-v)^{H_2-3/2}\Big)\Big)(s),$$
so, with the help of integration by parts formula, equation \eqref{A.6} can be rewritten as
\begin{equation}\begin{gathered}\label{A.7}C_5=\int_v^1\breve{h}(s)\Big(\mathcal{D}^{H_1-1/2}_{v+}
\Big(\cdot^{H_2-H_1}(\cdot-v)^{H_2-3/2}\Big)\Big)(s)ds\\=\int_v^1(\mathcal{D}^{H_1-1/2}_{1-}\breve{h})(s)s^{H_2-H_1}
(s-v)^{H_2-3/2} ds\\=(\Gamma(H_2-1/2))^{-1}(I_{1-}^{H_2-1/2}(\mathcal{D}^{H_1-1/2}_{1-}\breve{h}) \cdot^{H_2-H_1})(v),\end{gathered}\end{equation}
where $C_5=C_4(\beta_{H_2}\Gamma(3/2-H_1))^{-1}.$ The latter equation means that
$$(\mathcal{D}^{H_1-1/2}_{1-}\breve{h})(v) v^{H_2-H_1}=C_6(1-v)^{1/2-H_1}, $$
$C_6=\frac{C_5\Gamma(H_1-1/2)}{\Gamma(3/2-H_1)}.$ At last, we get that
 $$ {h} (v)=v^{H_1-1/2}\breve{h}(v)=C_6v^{H_1-1/2}(I^{H_1-1/2}_{1-}(\cdot^{H_1-H_2}(1-\cdot)^{1/2-H_1})(v),$$
 and this solution of equation \eqref{eqA.1} is unique.\end{proof}

 \begin{acknowledgement}
 My thanks to  M. Kleptsyna for the statement of the problem and to M. Kleptsyna and G. Shevchenko for helpful discussions.
 \end{acknowledgement}


\begin{thebibliography}{20}

 \bibitem {betotu} Bertin K.,  Torres S.,  Tudor C.: Drift parameter estimation in fractional diffusions driven by perturbed
random walks. Statistics \& Probability Letters, \textbf{81}, 243–249 (2011)
\bibitem{BiHuOksZhang} F. Biagini, Y. Hu, B. {\O}ksendal, T. Zhang Stochastic Calculus for Fractional Brownian Motion and Applications. Springer-Verlag, 2008.
\bibitem {bishwal} Bishwal, J. P. N.: Parameter estimation in stochastic differential equations
    \ Springer, Lecture Notes Math., \textbf{1923} (2008)
\bibitem{CaiChigKlept} C. Cai, P. Chigansky, M. Kleptsyna Mixed fractional  Brownian motion:
the filtering perspective.  arXiv:1208.6253v3 [math.PR] 23 September 2013
\bibitem{Decr} L. Decreusefond, A.S. $\ddot{U}$stunel Stochastic analysis of fractional Brownian motion. Potential analysis (10)1999, 177-214.

 \bibitem{HuNu}  Hu, Y.,  Nualart, D.: Parameter estimation for fractional Ornstein-Uhlenbeck processes.
 Statistics \& Probability Letters, \textbf{8}, 1030-1038 (2010)


     \bibitem{Jost} C. Jost Transformation formulas for fractional Brownian motion. Stoch. Processes and Appl. 116(2006), 1341-1357.
\bibitem{Mi} Yu. Mishura Stochastic calculus for fractional Brownian motion and related process. Lecture Notes Math., 1929, Springer-Verlag, 2008.
    \bibitem{NVV} I. Norros, E. Valkeila, J. Virtamo An elementary approach to a Girsanov formula and other analytical results on fractional Brownian motions. Bernoulli, 5(4)(1999), 571-587.

 \bibitem{Nu} D. Nualart   Stochastic integration with respect to fractional
Brownian motion and applications. Contemporary Mathematics, 336 (2003), 3–39.

\bibitem {prakasa rao} Prakasa Rao, B. L. S.: Statistical inference for fractional diffusion processes.
John Wiley \& Sons  (2010)
\bibitem{SamkoKM} S.G. Samko, A.A. Kilbas and O.I. Marichev Fractional integrals and derivatives: Theory and Applications. Gordon and Breach Science Publishers, London, 1993.






\end{thebibliography}
\end{document}